\documentclass{amsart}

\usepackage{amsmath,amssymb,amsfonts,amsthm,mathtools}
\usepackage{booktabs,array}
\usepackage{enumitem}
\usepackage[protrusion=true,expansion=false]{microtype}
\usepackage[hidelinks]{hyperref}
\newtheorem{theorem}{Theorem}[section]
\newtheorem{proposition}[theorem]{Proposition}
\newtheorem{lemma}[theorem]{Lemma}
\newtheorem{corollary}[theorem]{Corollary}
\newtheorem*{mainthm}{Main theorem}
\theoremstyle{definition}

\theoremstyle{remark}
\newtheorem{remark}[theorem]{Remark}

\newcommand{\A}{\mathbb A}
\newcommand{\R}{\mathbb R}
\newcommand{\Z}{\mathbb Z}
\newcommand{\Gm}{\mathbb G_m}
\newcommand{\KMW}{\mathbf K^{\mathrm{MW}}}
\newcommand{\KM}{\mathbf K^{\mathrm M}}
\newcommand{\Hcell}{\mathbf H^{\mathrm{cell}}}
\newcommand{\Ccell}{C^{\mathrm{cell}}}
\newcommand{\cell}{\mathrm{cell}}

\newcommand{\rank}{\operatorname{rank}}
\newcommand{\coker}{\operatorname{coker}}
\newcommand{\Spec}{\operatorname{Spec}}
\newcommand{\Hom}{\operatorname{Hom}}
\newcommand{\im}{\operatorname{im}}

\newcommand{\length}{\ell}
\newcommand{\htc}{\operatorname{ht}}
\newcommand{\GW}{\mathrm{GW}}

\title[Cellular \texorpdfstring{$\A^1$}{A1}-Homology of Split Flags]
{Cellular \texorpdfstring{$\A^1$}{A1}-Homology from Bruhat Boundary
Matrices of Split Semisimple Flag Varieties }
\author{Haoyang Liu and Tianle Liu}
\date{}
\address{University of California, Santa Barbara}
\email{haoyangliu@ucsb.edu}
\address{University of Southern California}
\email{tianleli@usc.edu}
\hypersetup{
  pdftitle={Cellular A1-Homology from Bruhat Boundary Matrices of Split Semisimple Flag Varieties},
  pdfauthor={Haoyang Liu and Tianle Liu}
}

\begin{document}

\begin{abstract}
We determine the cellular \(\A^1\)-chain complex of full and partial
flag varieties of split semisimple simply connected groups over a
perfect field of characteristic different from \(2\).
With compatible integral normal orientations and reindexing by
opposite Bruhat cells, the differential in each degree at least two
is represented by the motivic Hopf map \(\eta\) times one half of the
signed cellular boundary matrix of the corresponding real flag manifold.
Consequently, the Smith data of the resulting integer complex give
all positive-degree cellular homology sheaves in terms of Milnor--Witt \(K\)-theory and
the kernels and cokernels of integer multiples of \(\eta\).
We obtain a uniform computation of the second cellular \(\A^1\)-homology of \(SL_n/B\) for
\(n\geq3\), together with complete calculations for \(SL_3/B\) and
the full flags of types \(F_4\) and \(G_2\).
\end{abstract}

\maketitle
\tableofcontents
\enlargethispage{4pt}

\section{Introduction}
\label{sec:introduction}

In classical topology, the integral homology of a split real flag
manifold is computed from its Bruhat cells and their signed incidence
numbers.  These numbers are \(0\) or \(\pm2\) and form the signed
cellular boundary matrices.  The work of Kocherlakota and
Rabelo--San Martin gives descriptions of these matrices
\cite{Kocherlakota,RabeloSanMartin}.  More recent work of Lambert and
Rabelo provides explicit formulas in type \(A\) and computations in
classical and exceptional types
\cite{LambertRabeloTypeA,LambertRabelo2026}.

Over a perfect field, Morel--Sawant's cellular \(\A^1\)-homology gives
an analogous construction for smooth schemes with an oriented cellular
structure \cite{MorelSawant}.  Its chain groups are strictly
\(\A^1\)-invariant sheaves: an affine stratum of positive codimension
\(i\) contributes the Milnor--Witt \(K\)-theory sheaf \(\KMW_i\).
For a split full flag variety, Morel--Sawant construct this complex in
all degrees, identify it by torus augmentation, and compute its
degree-two differential and its homology in degrees zero and one
\cite[Theorem~5.8 and Corollary~5.9]{MorelSawant}.  They leave the
higher homology calculation to further work
\cite[Remark~5.10]{MorelSawant}.

The goal of this paper is to determine these differentials in every
degree and to compute the resulting homology sheaves for full and
partial flags.  We compare the motivic cellular boundary maps with
the cellular boundary maps with integer coefficients of the associated
split real flag manifold.  With compatible integral normal orientations,
the differential in degree at least two is represented by the motivic
Hopf map \(\eta\) times one half of the signed real cellular boundary
matrix, after reindexing by opposite Bruhat cells.  This comparison
holds over every perfect field of characteristic
different from \(2\), in all root types and ranks.

Let \(X_\Theta=G/P_\Theta\), where \(G\) is split, semisimple, and
simply connected, and \(P_\Theta\) is a standard parabolic subgroup.
Let \(W\) be the Weyl group and \(W_\Theta\) its parabolic subgroup.
Write
\[
 X_\Theta=\coprod_{w\in W^\Theta}C_w,\qquad
 C_w\simeq\A^{\ell(w)},\qquad d=\dim X_\Theta,
\]
where \(W^\Theta\) consists of the minimal representatives for
\(W/W_\Theta\).  The cellular complex is indexed by codimension:
\[
 \Ccell_i(X_\Theta)\cong
 \bigoplus_{\ell(w)=d-i}\KMW_i\quad(i\geq1),\qquad
 \Ccell_0(X_\Theta)\cong\Z.
\]
If \(w_0\) and \(w_\Theta\) are the longest elements of \(W\) and
\(W_\Theta\), then
\[
 \pi_\Theta(w)=w_0ww_\Theta,\qquad
 \ell(\pi_\Theta(w))=d-\ell(w).
\]
Let \(P_i\) denote the corresponding permutation matrix from the
codimension-\(i\) labels to the length-\(i\) labels.  Our main result
is the following comparison.

\begin{mainthm}
Let \(k\) be perfect of characteristic different from \(2\), and let
\(G\) be a split semisimple simply connected group over \(k\).
For a standard parabolic \(P_\Theta\), put \(X_\Theta=G/P_\Theta\).
Fix ordinary Bruhat cell orientations on the real flag manifold
associated to this split root datum, using the standard orientation
in dimension zero.  Let \(\delta_i^\Theta\) be its signed cellular
boundary matrices, and put
\[
 D_i^\Theta=\tfrac12\delta_i^\Theta,\qquad
 E_i^\Theta=P_{i-1}^{-1}D_i^\Theta P_i.
\]
Give the positive Bruhat cells the primitive integral normal
orientations induced by the opposite cells, with the degree
normalization of Section~\ref{sec:bruhat-cells}.  Then
\[
 \partial_i^{\A^1}=\eta E_i^\Theta\quad(i\geq2),\qquad
 \partial_1^{\A^1}=0.
\]
\end{mainthm}

The matrices \(\delta_i^\Theta\) depend only on the split root datum
and the chosen ordinary cell orientations.  Their coefficients belong
to \(\{0,\pm2\}\) by
\cite[Theorem~2.2 and Corollary~3.5]{RabeloSanMartin}, so
\(D_i^\Theta\) is an integer matrix.  Thus the same signed incidence
data determine the cellular differential over each permitted base
field.  The comparison is also compatible with projection to partial
flags: the map \(G/B\to G/P_\Theta\) induces the quotient of cellular
complexes which retains the maximal representatives in the parabolic
cosets (Proposition~\ref{lem:partial-flag-orientations}).

The proof uses the root curve joining two fixed points in a Bruhat
cover.  The union of the two corresponding cells is a smooth open
subscheme of a Schubert variety.  Purity expresses a component of the
cellular boundary map through the Thom connecting map of its normal bundle in
\(X_\Theta\), restricted to this curve.  The resulting Rost--Schmid residue is
determined by the transition of the normal determinant line.  On the
split flag scheme over \(\Z\), this transition has the form
\(\pm t^m\), so the coefficient is \(0\) or \(\pm\eta\) over \(k\).
The opposite real Bruhat filtration then determines all
signs simultaneously.  The use of normal bundles on root curves is related to
Matszangosz's calculation of real incidence coefficients in type \(A\)
\cite[\S3.3]{Matszangosz}.

Consequently, the homology sheaves can be read from a Smith
decomposition of the integer complex \(D_\bullet^\Theta\).  Let
\(b_i\) be its free homology rank, and let \(m_i(q)\) count its
elementary summands \(\Z\xrightarrow{q}\Z\) in degrees \(i+1,i\),
where \(q\geq1\).  The same unimodular changes of basis give a
noncanonical isomorphism
\[
 \begin{aligned}
 \Hcell_i(X_\Theta)
 &\cong(\KMW_i)^{b_i}
 \oplus\bigoplus_{q\geq1}(\KMW_i/q\eta)^{m_i(q)}\\
 &\qquad\oplus\bigoplus_{q\geq1}
 ({}_{q\eta}\KMW_i)^{m_{i-1}(q)}\qquad(i\geq1).
 \end{aligned}
\]
The two latter sheaves are the cokernel of
\(q\eta:\KMW_{i+1}\to\KMW_i\) and the kernel of
\(q\eta:\KMW_i\to\KMW_{i-1}\), respectively.
Here \(m_0(q)=0\), and \(\Hcell_0(X_\Theta)\cong\Z\).
The dependence on the ground field is contained in these sheaves;
the integer Smith data are obtained from one half of the signed real
Bruhat boundary matrices.
This extends Morel--Sawant's full-flag calculation to a homology
formula in every degree.  Section~\ref{sec:morel-sawant-comparison}
checks its agreement with their low-degree results, including the
symplectic case.

There are related computations in type \(A\).  Over infinite perfect
fields of characteristic different from two, Yang decomposes the
Milnor--Witt motives of Grassmannians and complete flags into sums of
\(\Z(j)[2j]\) and \(\Z/\eta(j)[2j]\), where \(\Z/\eta\) is the cone
of \(\eta\) \cite[Theorems~4.19 and~5.8]{YangSplit}.  His result
implies that the integral cohomology of real complete flags has only
\(2\)-torsion \cite[Remark~2.19]{YangSplit}.
Hudson--Matszangosz--Wendt compute the Chow--Witt rings of type-\(A\)
partial flags and establish the torsion statement for all their real
specializations \cite{HudsonMatszangoszWendt}.  The comparison here
takes place in the cellular complex of strictly
\(\A^1\)-invariant sheaves and applies to full and partial flags in
every root type.  It also gives additive Chow--Witt groups by applying
cellular cochains, as explained in Remark~\ref{rem:type-a-wendt}.
Related cellular \(\A^1\)-homology computations for smooth toric
varieties appear in \cite{LiuPeng}.

When the associated real flag manifold has only \(2\)-torsion in
integral homology, the formula simplifies to summands of
\(\KMW_i\), \(\KM_i\), and \({}_{\eta}\KMW_i\)
(Corollary~\ref{cor:two-torsion}).  For example, the type-\(A\)
torsion result and the real Betti numbers give
\[
 \Hcell_2(SL_n/B)\cong
 (\KM_2)^{n(n-3)/2}\oplus({}_{\eta}\KMW_2)^{n-1}
 \qquad(n\geq3).
\]
We prove this in Corollary~\ref{cor:type-a-second-homology} and give
complete homology calculations for \(SL_3/B\) and the full flags of
types \(F_4\) and \(G_2\).  The available computations of boundary matrices used
here include type \(A\), classical types \(B_n,C_n,D_n\) with
\(n\leq7\), and exceptional types \(F_4,E_6,E_7\)
\cite{LambertRabeloTypeA,LambertRabelo2026}.  Each such computation
supplies integer input for the general comparison.

In our subsequent work \cite{LiuLiuTorus}, we retain the torus
translation data in the cellular boundary maps of \(G/U\) over fields of
characteristic zero.  As a further application,
Section~\ref{sec:torus-support} applies the Smith formula to the
associated graded complexes of its torus-support filtration.
Proposition~\ref{prop:support-smith} describes the \(E^1\)-page,
and Corollary~\ref{cor:support-bottom-row} computes its bottom row
and proves that the differentials \(d^1\) into its positive-support
terms are surjective.

The paper is organized as follows.  Sections~\ref{sec:bruhat-cells}
and~\ref{sec:real-boundary} fix the orientations and recall the real
Bruhat boundary matrices.  Section~\ref{sec:comparison} proves the comparison
and its compatibility with partial flags.  Section~\ref{sec:smith}
derives the homology formulas and their applications.
Sections~\ref{sec:sl3-example}--\ref{sec:g2-example} contain the
calculations for \(SL_3/B\), \(F_4\), and \(G_2\).

\section{Bruhat cells and oriented cellular chains}
\label{sec:bruhat-cells}

Throughout the paper, \(k\) is a perfect field of characteristic different
from \(2\).  Let \(G\) be a split semisimple simply connected group over
\(k\), with a fixed split pinning, maximal torus \(T\), and Borel
subgroup \(B\).  Write \(B^-\) for the opposite Borel, \(\Phi^+\) for
the positive roots, and \(\rho=\frac12\sum_{\alpha\in\Phi^+}\alpha\).
For a subset \(\Theta\) of the simple roots, let \(P_\Theta\supset B\)
be the standard parabolic, and put \(X_\Theta=G/P_\Theta\) and
\(d=\dim X_\Theta\).  Let \(W_\Theta\subset W\) be the corresponding
parabolic Weyl group and
\[
 W^\Theta=\{w\in W:w(\Phi_\Theta^+)\subset\Phi^+\}.
\]
Its elements are the minimal representatives for \(W/W_\Theta\).
The positive and opposite Bruhat cells are
\[
 C_w=B\dot wP_\Theta/P_\Theta\simeq\A^{\ell(w)},\qquad
 C_w^-=B^-\dot wP_\Theta/P_\Theta\simeq\A^{d-\ell(w)}.
\]
We study \(\Hcell_i(X_\Theta)\) in the abelian category
\(\mathbf{Ab}_{\A^1}(k)\) of strictly \(\A^1\)-invariant Nisnevich sheaves.

Fix a pinned split Chevalley group \(G_{\Z}\) with the prescribed root
datum, and form \(X_{\Theta,\Z}=G_{\Z}/P_{\Theta,\Z}\).  Choose integral
root homomorphisms \(x_\alpha:\mathbb G_{a,\Z}\to U_\alpha\) and Weyl
representatives.  For \(w\in W^\Theta\), if
\(\mathbf w=(s_{i_1},\ldots,s_{i_r})\) is a reduced word for \(w\), put
\(\beta_j=s_{i_1}\cdots s_{i_{j-1}}(\alpha_{i_j})\).
The ordered root product gives the usual integral chart
\[
 \Phi_{\mathbf w}:\A^r_{\Z}\xrightarrow{\sim} C_{w,\Z},\qquad
 (a_1,\ldots,a_r)\longmapsto
 x_{\beta_1}(a_1)\cdots x_{\beta_r}(a_r)\dot wP_{\Theta,\Z}.
\]
The root vectors are primitive integral generators.  Signs of the
ordered determinant generators are chosen to agree, over \(\R\), with
the ordinary cell orientations used in the signed boundary matrices.
The unique zero-dimensional ordinary cell has its standard orientation.
For a minimal representative, the full and partial cell charts are
chosen using the same ordered root coordinates, so their orientations
agree under \(G_{\Z}/B_{\Z}\to X_{\Theta,\Z}\).

Let \(w_0\) and \(w_\Theta\) be the longest elements of \(W\) and
\(W_\Theta\).  Left multiplication by a fixed integral representative
\(\dot w_0\) identifies
\[
 C_{w,\Z}^-\xrightarrow{\sim}C_{\pi_\Theta(w),\Z},\qquad
 \pi_\Theta(w)=w_0ww_\Theta.
\]
Orient the tangent determinant of \(C_{w,\Z}^-\) by pulling back the
chosen ordinary cell orientation through this isomorphism.  At the fixed
point \(wP_\Theta\), the two tangent spaces are complementary integral
root submodules.  Consequently the differential gives an isomorphism
\begin{equation}
 \label{eq:opposite-normal-orientation}
 T_wC_{w,\Z}^-\xrightarrow{\sim}
 N_{C_{w,\Z}/X_{\Theta,\Z}}|_w.
\end{equation}
We fix the Thom sphere identifications and their real generators once
in every degree, independently of the flag variety or cell.  We use
the normalization in which multiplication by \(\eta\) realizes to
\(+2\) on the corresponding cellular generators
\cite[Remark~2.50]{MorelSawant}.  This fixes the suspension signs as
well as the sphere generators.  For comparison, the sphere convention
of \cite[Proposition~3.3]{AsokFaselWilliams} gives \(-2\); replacing
its real generator in degree \(i\) by \((-1)^i\) times that generator
gives the normalization used here.

We now specify the normal orientation in positive codimension without
identifying these sphere generators with an uncalibrated coordinate-disk
orientation.  Choose the sign of the primitive normal determinant
generator at \(w\) so
that its Thom generator, with the fixed normalization, realizes under
\eqref{eq:opposite-normal-orientation} to the oriented generator of
\(C_w^-(\R)\).  This condition determines the sign uniquely: changing
the determinant frame by \(-1\) acts on the real Thom class by
\(-1\).  The adjustment to the raw tangent determinant depends only
on the fixed sphere convention in that codimension.  Thus all signs
are fixed before considering any Bruhat cover.

The resulting primitive generator extends uniquely to a determinant
trivialization on \(C_{w,\Z}\): its Picard group is zero, and the
units of \(\Z[a_1,\ldots,a_r]\) are \(\{\pm1\}\).
Base change gives the normal orientations over \(k\) used throughout.
The real normalization selects integral signs on the universal split
flag scheme, so it imposes no condition on the ground field \(k\).
In codimension zero, the unique open cell contributes \(\Z\) with its
canonical generator; no determinant-frame sign is chosen.

A change of normal determinant frame in positive codimension by
\(a\in k^\times\) acts on the
Thom generator by \(\langle a\rangle\).  This is the orientation
convention in the cellular formalism of Morel--Sawant
\cite[\S2]{MorelSawant}.  We distinguish these normal orientations
from the ordinary tangent orientations of the cells.  Their relation
is the opposite-cell identification
\eqref{eq:opposite-normal-orientation} with the fixed Thom
normalization just described.

\begin{lemma}
\label{lem:bruhat-a1-cellular}
With the preceding normal orientations,
\[
 \Ccell_i(X_\Theta)\cong
 \bigoplus_{\substack{w\in W^\Theta\\\ell(w)=d-i}}\KMW_i
 \quad(i\geq1),\qquad
 \Ccell_0(X_\Theta)\cong\Z.
\]
\end{lemma}

\begin{proof}
Set \(\Omega_{-1}=\varnothing\), and, for \(0\leq i\leq d\), set
\[
 \Omega_i=X_\Theta\setminus
 \bigcup_{\ell(w)\leq d-i-1}\overline{C_w}.
\]
The Bruhat closure relation gives
\[
 \Omega_i\setminus\Omega_{i-1}
 =\coprod_{\ell(w)=d-i}C_w.
\]
These are smooth closed strata in \(\Omega_i\), each isomorphic to an
affine space.  Their normal bundles are trivial over \(k\), and their
determinant orientations have been fixed above.  Homotopy purity and
the oriented Thom identifications in \cite[\S2]{MorelSawant} identify
the summand of a codimension-\(i\) cell with \(\KMW_i\) for \(i>0\).
The unique open cell gives \(\Z\) in degree zero.
\end{proof}

The complex is therefore indexed by codimension.  Changing the
primitive normal determinant generator on a positive-codimension cell
changes its Thom
generator by \(\langle-1\rangle\).  The convention above is needed to
identify the resulting signed matrix with the prescribed ordinary
Bruhat boundary matrix.

\section{The integral Bruhat boundary matrix}
\label{sec:real-boundary}

Let \(A_i^\Theta\) be the free abelian group generated by the Bruhat cells
of dimension \(i\):
\[
  A_i^\Theta = \bigoplus_{\substack{w\in W^\Theta\\ \length(w)=i}}\Z[w].
\]
Let \(X_{\Theta,\R}\) be the split real flag manifold with the
same root datum.  Its cellular boundary maps with integer coefficients are
\[
  \delta_i^\Theta\colon A_i^\Theta\longrightarrow A_{i-1}^\Theta .
\]
We use the same notation for their signed matrices in the cellular bases
with the orientations fixed in Section~\ref{sec:bruhat-cells}.
For a Bruhat cover \(w\gtrdot w'\), write the coefficient of \([w']\) in
\(\delta_i^\Theta[w]\) as \(c_\Theta(w,w')\).
These integer matrices are defined for every root type and parabolic.
In this section \(\Psi_{\mathbf w}\) denotes
the ordinary real characteristic parametrization attached to a reduced
word, and \(\deg_{\mathrm{coord}}\) denotes the signed degree in the
formula for the real boundary coefficients.  These degrees are integer input; we do not
interpret the real characteristic parametrizations as algebraic maps
over an arbitrary field.
Fix one such parametrization for each cell and give that cell the
induced parameter orientation.  Write \(c^\Psi(w,w')\) for the resulting
full-flag boundary coefficients.  If the chosen ordinary orientation
of \(C_w(\R)\) is \(\varepsilon_w\) times its parameter orientation, then
\(\varepsilon_w\in\{\pm1\}\) and \(\varepsilon_e=1\).

\begin{proposition}
\label{prop:lr-boundary}
For the full flag, fix a reduced expression \(w=s_1\cdots s_i\).  Suppose the Bruhat cover
\(w'\) is obtained, up to the fixed normal form of \(w'\), by deleting the
\(I\)-th reflection:
\[
  w'=s_1\cdots \widehat{s_I}\cdots s_i .
\]
Let \(\gamma\) be the root such that \(w=w's_\gamma\).  Then
\[
  c^\Psi(w,w')
  =
  (-1)^I\,
  \deg_{\mathrm{coord}}\!\left(\Psi_{w'}^{-1}\circ \Psi_{\widehat w_I}\right)
  \bigl(1+(-1)^{\htc(\gamma^\vee)}\bigr).
\]
For the chosen ordinary cell orientations,
\[
 c(w,w')=\varepsilon_w\varepsilon_{w'}c^\Psi(w,w').
\]
In particular \(c(w,w')\in\{0,\pm 2\}\).  By the chain-level projection
theorem for partial flags, the partial-flag cellular boundary map is
obtained from the full-flag map by restricting to the minimal
coset representatives \(W^\Theta\) and projecting back to that span.
The parameter orientations and the signs \(\varepsilon_w\) on minimal
representatives are taken compatibly with the partial-flag charts.
Here \(\htc(\gamma^\vee)\) means the height of the coroot in the dual root
system, with the convention of \cite{LambertRabelo2026}.
\end{proposition}

\begin{proof}
The formula for the parameter orientations follows from
\cite[Theorem~2.8]{RabeloSanMartin} and the coroot-height
description of \cite{LambertRabelo2022}.
The factor \(1+(-1)^{\htc(\gamma^\vee)}\) forces the
coefficient to vanish or have absolute value \(2\).  The remaining sign is
the product of the deletion sign \((-1)^I\) and the degree of the change of
coordinates between the reduced decomposition induced by deletion and the
chosen normal form.  We write it as \(\deg_{\mathrm{coord}}\) to emphasize
that this is the ordinary real degree of the coordinate change.
Changing the source and target cell generators multiplies a boundary
coefficient by \(\varepsilon_w\varepsilon_{w'}\).
The coefficient bound in every type is
\cite[Theorem~2.2]{RabeloSanMartin}; the passage to partial flags
is \cite[Theorem~3.4 and Corollary~3.5]{RabeloSanMartin}.
\end{proof}

The explicit formulas and sign computations in
\cite{LambertRabeloTypeA,LambertRabelo2026} supply these matrices
for their tabulated cases.  Their scope concerns the available
computations; the ordinary cellular boundary matrices and the coefficient
bound above are defined for every split root datum.

Since all non-zero coefficients in \(\delta_i^\Theta\) are even,
define
\[
  D_i^\Theta = \frac{1}{2}\delta_i^\Theta.
\]
Because \((\delta^\Theta)^2=0\) and the groups are torsion-free,
\((D^\Theta)^2=0\).  Thus \(D_\bullet^\Theta\) is itself an integral chain
complex.
In degree one, \(D_1^\Theta=0\).  Indeed, every length-one cover is
obtained from a simple reflection, and the coefficient factor in
Proposition~\ref{prop:lr-boundary} is
\[
  1+(-1)^{\htc(\alpha^\vee)}=1+(-1)=0 .
\]
The restrict-and-project construction for partial flags preserves this
vanishing.

The Morel--Sawant cellular complex is indexed by codimension.  The
ordinary real Bruhat cellular complex, however, is indexed by the
dimension \(\ell(w)\) of the Bruhat cell.  The comparison between the two
uses the length-reversing map below.  Let \(w_0\) be the longest element of
\(W\), let \(w_\Theta\) be the longest element of \(W_\Theta\), and define
\[
  \pi_\Theta(w)=w_0ww_\Theta\in W^\Theta .
\]
Then \(\ell(\pi_\Theta(w))=d-\ell(w)\), and \(\pi_\Theta\) reverses the
Bruhat order on \(W^\Theta\).  Let
\[
  P_i\colon
  \bigoplus_{\ell(w)=d-i}\Z[w]\longrightarrow
  \bigoplus_{\ell(u)=i}\Z[u],
  \qquad
  P_i[w]=[\pi_\Theta(w)]
\]
be the corresponding permutation matrix.  The integral matrix controlling
the codimension-indexed motivic differential from degree \(i\) to degree
\(i-1\) is
\[
  E_i^\Theta := P_{i-1}^{-1}D_i^\Theta P_i .
\]
Equivalently, the coefficient of \([v]\), with
\(\ell(v)=\ell(w)+1\), in \(E_i^\Theta[w]\) is
\[
  \frac{1}{2}c_\Theta(\pi_\Theta(w),\pi_\Theta(v)).
\]

The indexing dictionary is:
\[
\begin{array}{c|c|c}
\text{ordinary real cellular degree} & \text{Bruhat length}
& \text{cellular $\mathbb{A}^1$-degree}\\
\hline
i & \ell(u)=i & \ell(w)=d-i\\
i\to i-1 & D_i^\Theta & E_i^\Theta=P_{i-1}^{-1}D_i^\Theta P_i
\end{array}
\]
Thus \(E_i^\Theta\) is the matrix \(D_i^\Theta\) transported through
the permutation \(P_i\).  In this convention the projective-plane case
\(\mathbb P^2\simeq SL_3/P\) has the differential
\(\KMW_2\to\KMW_1\) supplied by the corresponding top-to-middle entry
of the real cellular boundary matrix.

\section{The Bruhat--Milnor--Witt comparison}
\label{sec:comparison}

We compute a component of the cellular boundary map by restricting a normal bundle to the
root curve joining two adjacent fixed points.  The parity of its
determinant degree determines whether the component vanishes.  Integral
determinant frames constrain the remaining coefficient to
\(\eta\) or \(-\eta\); the opposite Bruhat filtration then determines
the sign.

For a unit \(a\), recall the identities
\[
 \langle a\rangle=1+\eta[a],\qquad
 (1+\langle-1\rangle)\eta=0,
 \qquad \langle-1\rangle\eta=-\eta.
\]
We use the Rost--Schmid residue normalized by
\(\partial_t[t]=1\); its rules for multiplication by units and by
\(\eta\) are recalled in \cite[Theorem~1.7 and Lemma~1.10]{Fasel}.
The cellular cochain description and its compatibility with purity
are as in \cite[\S2]{MorelSawant}.

\begin{lemma}
\label{lem:determinant-residue}
Let \(E\) be a vector bundle of rank \(r\geq1\) on \(\mathbb P^1_k\),
with the two-stratum filtration \(U\sqcup\{p\}\), where
\(U\simeq\A^1\).  Fix determinant frames on the two standard affine
charts and a parameter \(t\) at \(p\).  Suppose their transition is
\(ct^m\), with \(c\in k^\times\).  Use the Thom identifications
compatible with the twisted residue convention \(\partial_t[t]=1\).
Then the two-term reduced cellular differential of
\(\operatorname{Th}(E)\) is
\begin{equation}
 \label{eq:determinant-residue}
 \KMW_{r+1}\longrightarrow\KMW_r,
 \qquad
 \begin{cases}
 0,&m\text{ even},\\
 \langle c\rangle\eta,&m\text{ odd}.
 \end{cases}
\end{equation}
The closed Thom identification includes the tangent-line twist in the
residue, as in \cite[\S1.5]{Fasel}.  Passing to the globally fixed
Thom identifications of Section~\ref{sec:bruhat-cells} introduces only
the fixed signs from ordering determinant factors and identifying
suspensions.  In particular, if \(c\in\{1,-1\}\), the coefficient
in those identifications belongs to \(\{0,\eta,-\eta\}\), and
vanishes precisely when \(m\) is even.  The parity of \(m\) is the
parity of \(\deg\det E\).
\end{lemma}

\begin{proof}
Apply cellular cochains with coefficients in \(\KMW_r\) to the
two Thom summands.  Under contraction, the identity of the lower
summand corresponds to
\(1\in(\KMW_r)_{-r}(k)=\GW(k)\).  In the determinant frame near
\(p\), this class is \(\langle ct^m\rangle\).  The connecting
homomorphism is its Rost--Schmid residue
\[
 \partial_t:\KMW_0(k(t))\longrightarrow\KMW_{-1}(k).
\]
If \(m=2a\), the class is \(\langle c\rangle\), whose residue is
zero.  If \(m=2a+1\), it is
\(\langle c\rangle\langle t\rangle
 =\langle c\rangle(1+\eta[t])\), whose residue is
\(\langle c\rangle\eta\).  These identities hold also for negative
integers \(a\), since square factors do not change a one-dimensional
form.  Finally,
\[
 \Hom_{\mathbf{Ab}_{\A^1}(k)}(\KMW_{r+1},\KMW_r)
 \cong(\KMW_r)_{-(r+1)}(k)=\KMW_{-1}(k)
\]
identifies this cochain calculation with the chain differential.
\end{proof}

The Bruhat order on \(W^\Theta\) is induced from \(W\) and graded
by length.  Thus strong exchange writes every cover as
\(v=ws_\gamma\) for a positive root \(\gamma\).

\begin{lemma}
\label{lem:root-curve-reduction}
Let \(w,v\in W^\Theta\), with \(v=ws_\gamma\gtrdot w\) and
\(\gamma\in\Phi^+\), and suppose \(\operatorname{codim}C_w=i\geq2\).
Put \(\beta=w\gamma\).  Then
\[
 Y=C_v\sqcup C_w\subset\overline{C_v}
\]
is smooth, and the root curve
\[
 C=U_\beta\cdot vP_\Theta\sqcup\{wP_\Theta\}
 \simeq\mathbb P^1
\]
lies in \(Y\) and is transverse to \(C_w\) at \(wP_\Theta\).
For \(E=N_{Y/X_\Theta}\), the component
\(\mu(v,w):\KMW_i\to\KMW_{i-1}\) of the cellular differential
is the two-term Thom differential of \(E_C=E|_C\), with the induced
normal determinant orientations.
\end{lemma}

\begin{proof}
Since \(w\) is a coatom of \(v\), the union of the other strata in
\(\overline{C_v}\) is closed, so \(Y\) is open in that Schubert
variety.  The finite-dimensional Schubert varieties in the split
flag scheme are geometrically normal in every characteristic
\cite[\S4]{HainesLourencoRicharz}; hence they are regular at the
generic points of their codimension-one strata.
Over an algebraic closure, smoothness is invariant under \(B\), which
acts transitively on \(C_w\).  Thus the entire stratum \(C_w\), and
therefore \(Y\), is smooth.  The same argument applies to the
geometric fibers of the integral Schubert schemes.

The rank-one root subgroup for \(\beta\) gives the displayed root
curve.  The cover condition implies \(\beta>0\), so its affine part
lies in \(C_v\), and its other endpoint is \(wP_\Theta\).
At \(w\) its tangent is \(\mathfrak g_{-\beta}\), whereas the
tangent weights of \(C_w\) are positive roots.  This direction is
therefore complementary to \(T_wC_w\) in \(T_wY\).

To identify the differential, let \(M\) be the open subset of
\(\Omega_i\) obtained by deleting the codimension-\(i\) cells other
than \(C_w\).  Both \(Z=C_w\) and \(Y\) are closed and smooth in
\(M\).  The inclusions of supports \(Z\subset Y\subset M\) give
the cofiber sequence
\begin{equation}
 \label{eq:nested-supports}
 (M\setminus Z)/(M\setminus Y)
 \longrightarrow M/(M\setminus Y)
 \longrightarrow M/(M\setminus Z).
\end{equation}
Purity \cite[\S3, Theorem~2.23]{MorelVoevodsky} identifies its terms with
\[
 \operatorname{Th}(E|_{C_v})\longrightarrow
 \operatorname{Th}(E)\longrightarrow
 \operatorname{Th}(N_{Z/M}).
\]
For precision, apply naturality of purity to the open restriction of
\(Y\subset M\) over \(M\setminus Z\), and then take cofibers of
the resulting square.  If \(\operatorname{Tot}(E)\) is the total
space, the cofiber of the first Thom map is
\[
 \operatorname{Tot}(E)/(\operatorname{Tot}(E)\setminus Z)
 \simeq\operatorname{Th}(E|_Z\oplus N_{Z/Y}),
\]
where \(Z\) lies in the zero section.  This direct sum is supplied
by the tangent splitting along that zero section; it does not require
a splitting of \(N_{Z/M}\).  The normal exact sequence
\[
 0\longrightarrow N_{Z/Y}\longrightarrow N_{Z/M}
 \longrightarrow E|_Z\longrightarrow0
\]
identifies its determinant with
\(\det N_{Z/Y}\otimes\det(E|_Z)\).  Reordering the factors as
\(\det(E|_Z)\otimes\det N_{Z/Y}\) introduces
\((-1)^{i-1}\); we use this ordered determinant isomorphism to
orient the closed Thom term.  This is the determinant comparison
induced by purity: in compatible \(\acute{e}\)tale coordinates,
\(Y\) is cut out by \(x_1,\ldots,x_{i-1}\), and \(Z\) inside
\(Y\) by \(t\), so both constructions use the ordered normal
coordinates \((x_1,\ldots,x_{i-1},t)\).  Naturality of purity
under \(\acute{e}\)tale maps makes these local identifications
compatible.  The cofiber comparison therefore respects both the
induced determinant orientations and the connecting morphism.
Excision and the projections onto
the \(C_w\) and \(C_v\) summands of the cellular filtration identify
that connecting morphism with \(\mu(v,w)\).

Pull back the Thom filtration along \(C\hookrightarrow Y\).
Transversality at \(w\) identifies its closed term with
\(\operatorname{Th}(E_w\oplus T_wC)\).  It gives a morphism from
\[
 \operatorname{Th}(E_C|_{C\setminus\{w\}})\longrightarrow
 \operatorname{Th}(E_C)\longrightarrow
 \operatorname{Th}(E_w\oplus T_wC)
\]
to the preceding sequence.  On the two end terms, the induced Thom
homology maps are isomorphisms: their base maps are
\(\A^1\to C_v\) and \(\Spec k\to C_w\), between
\(\A^1\)-contractible spaces, with the pulled-back determinant
orientations.  Naturality of the connecting morphism gives the claim.
\end{proof}

Write \(\rho_\Theta=\frac12\sum_{\alpha\in\Phi_\Theta^+}\alpha\).
The next calculation applies directly to partial flags.

\begin{lemma}
\label{lem:normal-determinant-degree}
In the situation of Lemma~\ref{lem:root-curve-reduction},
\begin{equation}
 \label{eq:partial-normal-degree}
 \deg\det E_C
 =\langle\rho-2\rho_\Theta,\gamma^\vee\rangle-1
 =\htc((w_\Theta\gamma)^\vee)-1.
\end{equation}
For \(G/B\), this degree is \(\htc(\gamma^\vee)-1\).
\end{lemma}

\begin{proof}
Put \(h=\langle\rho,\gamma^\vee\rangle\).  For a
\(T\)-equivariant line bundle \(L\) on \(C\), let \(\lambda_v\)
and \(\lambda_w\) be its endpoint weights.  With the root-curve
coordinate of weight \(\beta\) at \(v\),
\[
 \lambda_v-\lambda_w=(\deg L)\beta.
\]
This sign convention gives degree \(2\) for \(TC\), whose endpoint
weights are \(\beta\) and \(-\beta\).

The tangent weights of the cell \(C_u\) sum to \(\rho-u\rho\).
At \(v\), this is the determinant weight of \(TY\); at \(w\),
one adds the transverse weight \(-\beta\).  Since
\(v\rho=w\rho-h\beta\),
\[
 \operatorname{wt}(\det T_vY)-\operatorname{wt}(\det T_wY)
 =(\rho-v\rho)-(\rho-w\rho-\beta)=(h+1)\beta.
\]
Thus \(\deg\det(TY|_C)=h+1\).
The ambient tangent determinant at \(uP_\Theta\) has weight
\(-u(2\rho-2\rho_\Theta)\), so
\[
 \deg\det(TX_\Theta|_C)
 =\langle2\rho-2\rho_\Theta,\gamma^\vee\rangle.
\]
Subtracting the degree of \(TY|_C\) proves the first equality.
For the second, use \(w_\Theta\rho=\rho-2\rho_\Theta\).
Here \(\gamma\notin\Phi_\Theta\), because the two endpoints are
distinct in \(W/W_\Theta\), and \(w_\Theta\) preserves the positive
roots outside \(\Phi_\Theta\).  Thus \(w_\Theta\gamma\) is positive
and its coroot height is the displayed pairing.
\end{proof}

\begin{lemma}
\label{lem:integral-determinant}
With the integral normal orientations of Section~\ref{sec:bruhat-cells},
the determinant transition used in
Lemma~\ref{lem:determinant-residue} is \(\varepsilon t^m\), where
\(\varepsilon\in\{1,-1\}\) and
\(m\equiv\deg\det E_C\pmod2\).  The sign is independent of the
base field.  In particular,
\[
 \mu(v,w)\in\{0,\eta,-\eta\}.
\]
\end{lemma}

\begin{proof}
Work over \(\Z\).  Integral Schubert schemes are flat, of finite
presentation, and have geometrically normal fibers
\cite[\S4]{HainesLourencoRicharz}.  Their two-stratum open
subschemes \(Y_\Z=C_{v,\Z}\sqcup C_{w,\Z}\) therefore have smooth
geometric fibers by the codimension-one argument in
Lemma~\ref{lem:root-curve-reduction}.  Flatness and the fiberwise
criterion for smoothness show that \(Y_\Z\) is smooth over \(\Z\).
Its immersion in the smooth flag scheme is regular, so its relative
normal bundle is an integral vector bundle.  The integral root curve
\(C_\Z\simeq\mathbb P^1_\Z\) and its transverse tangent line are
defined by the integral rank-one root subgroup and primitive
Chevalley root coordinates.  Its tangent line maps to a primitive
root summand, so this transversality is preserved under every
base change.

At \(w\), transversality yields an exact sequence of free integral
modules
\[
 0\longrightarrow T_wC_\Z
 \longrightarrow N_{C_{w,\Z}/X_{\Theta,\Z}}|_w
 \longrightarrow E_w\longrightarrow0.
\]
The first term is a primitive root summand.  Taking determinants
induces a primitive generator of \(\det E_w\) from the chosen
generator of \(\det N_{C_w/X_\Theta}|_w\) and the root parameter.
Permuting determinant factors changes this generator by an integral
sign.  At \(v\), \(E_v=N_{C_v/X_\Theta}|_v\) has its chosen
primitive generator.

The line bundle \(L=\det(E|_{C_\Z})\) is trivial on each standard
affine chart: \(\operatorname{Pic}(\A^1_\Z)=0\).  Choose frames
whose values at the respective endpoints are these primitive
generators.  On the open-cell chart, the resulting frame agrees with
the restriction of the chosen normal determinant frame of \(C_v\),
since the only units of \(\Z[t]\) are \(\pm1\).  The two frames
have transition in
\[
 \Z[t,t^{-1}]^\times=\{\varepsilon t^m:
 \varepsilon\in\{1,-1\},\ m\in\Z\}.
\]
Its exponent has the parity of \(\deg L\), regardless of the
direction of transition.  Base change and
Lemma~\ref{lem:determinant-residue} give either zero or
\(\langle\varepsilon\rangle\eta=\varepsilon\eta\) in the residue
identifications.  The fixed ordering and suspension signs needed for
the global Thom identifications are also integral signs and are
independent of the field.  They preserve the set
\(\{0,\eta,-\eta\}\), proving the assertion.
\end{proof}

\begin{proposition}
\label{prop:opposite-filtration}
Over \(\R\), the relative homology complex of the open filtration
\(\Omega_\bullet(\R)\), with the normalized Thom generators fixed
in Section~\ref{sec:bruhat-cells}, is the ordinary Bruhat cellular
complex transported by \(P_\bullet\).  Its differential is represented by
\[
 P_{i-1}^{-1}\delta_i^\Theta P_i.
\]
\end{proposition}

\begin{proof}
Let \(F_i^-\) be the \(i\)-skeleton of the opposite real Bruhat
decomposition.  The Richardson incidence relation gives
\(C_x^-\cap C_z\ne\varnothing\) only if \(x\leq z\).
If \(\dim C_x^-\leq i\), then \(\ell(x)\geq d-i\), so every
positive cell meeting \(C_x^-\) lies in \(\Omega_i\).  Thus
\(F_i^-\subset\Omega_i(\R)\), compatibly with the filtrations.
There are induced maps
\[
 H_i(F_i^-,F_{i-1}^-;\Z)\longrightarrow
 H_i(\Omega_i(\R),\Omega_{i-1}(\R);\Z).
\]
The target generators are the normal Thom classes of positive cells of
codimension \(i\).  An opposite \(i\)-cell
\(C_x^-\) can meet such a cell \(C_z\) only when \(x=z\), since
the two labels have equal length.  Their intersection is then the
single fixed point \(x\), with complementary tangent root spaces.
The coefficient is \(+1\) by the defining normalization of the
normal Thom generator through
\eqref{eq:opposite-normal-orientation}.  Hence these maps identify all
generators simultaneously.  Naturality of the relative connecting
homomorphisms makes them an isomorphism of chain complexes.
Finally, \(\dot w_0\) identifies the oriented opposite cells with
the chosen ordinary cells, with label permutation \(\pi_\Theta\).
This gives the stated matrix.
\end{proof}

\begin{theorem}
\label{thm:eta-comparison}
Let \(X_\Theta=G/P_\Theta\) be as in
Section~\ref{sec:bruhat-cells}, with the normal orientations fixed
there.  For
\(i\geq2\), \(\ell(w)=d-i\), and \(v\gtrdot w\),
\begin{equation}
 \label{eq:signed-comparison}
 \mu(v,w)=\frac{c_\Theta(\pi_\Theta(w),\pi_\Theta(v))}{2}\eta.
\end{equation}
Consequently
\[
 \partial_i^{\A^1}=\eta P_{i-1}^{-1}D_i^\Theta P_i
 =\eta E_i^\Theta\quad(i\geq2),\qquad
 \partial_1^{\A^1}=0.
\]
\end{theorem}

\begin{proof}
Lemmas~\ref{lem:root-curve-reduction} and
\ref{lem:integral-determinant} express every component as
\(a(v,w)\eta\), where \(a(v,w)\in\{0,1,-1\}\) is an integer.
The Thom cofiber sequence in Lemma~\ref{lem:root-curve-reduction}
represents this component by an actual connecting map
\[
 \kappa_{v,w}:S_s^i\wedge\Gm^{\wedge i}
 \longrightarrow
 \Sigma_s\bigl(S_s^{i-1}\wedge\Gm^{\wedge(i-1)}\bigr)
 =S_s^i\wedge\Gm^{\wedge(i-1)}.
\]
Here \(S_s^i\) and \(\Sigma_s\) denote the simplicial sphere and
simplicial suspension.  Since \(i\geq2\), the target is
\((i-1)\)-\(\A^1\)-connected.  The first Hurewicz isomorphism
and \(i\)-fold contraction identify the group of such maps with
\[
 (\KMW_{i-1})_{-i}(k)=\KMW_{-1}(k),
\]
as in \cite[Lemma~3.1]{AsokFaselWilliams}.  By naturality of
Hurewicz and contraction, this identification is the map induced by
\(\A^1\)-homology to
\(\Hom(\KMW_i,\KMW_{i-1})\): both send a sphere map to the
image of its fundamental class.  Thus the class of
\(\kappa_{v,w}\) is the same element \(a(v,w)\eta\) computed
above.

For the same integral Thom data over \(\R\), real realization of
this sphere map has degree \(2a(v,w)\), by
\cite[Proposition~3.3]{AsokFaselWilliams} and the generator
normalization of Section~\ref{sec:bruhat-cells}.  Naturality of
realization for the cofiber sequence identifies this degree with
the relative boundary coefficient.  By
Proposition~\ref{prop:opposite-filtration}, it is also
\(c_\Theta(\pi_\Theta(w),\pi_\Theta(v))\).
The integer \(a(v,w)\) is independent of the base field, so dividing
this equality by \(2\) proves \eqref{eq:signed-comparison} over every
permitted base field.  This use of realization identifies an already
integral sign; it does not require injectivity of the signature on
\(W(k)\).

Only Bruhat covers contribute to the adjacent cellular differential,
so the coefficient formula gives the stated matrix.  Finally,
\(\Hom_{\mathbf{Ab}_{\A^1}(k)}(\KMW_1,\Z)=\Z_{-1}(k)=0\), which
proves the assertion in degree one.
\end{proof}

The determinant computation also recovers the height factor in the
ordinary coefficient formula.  For \(v=ws_\gamma\gtrdot w\), put
\(\gamma'=w_\Theta\gamma\).  Then
\[
 \pi_\Theta(w)=\pi_\Theta(v)s_{\gamma'},\qquad
 \deg\det E_C=\htc((\gamma')^\vee)-1.
\]
The residue is nonzero precisely when this coroot height is even.
Put \(a=\pi_\Theta(w)\) and \(b=\pi_\Theta(v)\), and let
\(\varepsilon_a,\varepsilon_b\) be the orientation signs of
Proposition~\ref{prop:lr-boundary}.  If \(I\) and
\(\sigma\in\{\pm1\}\) are the deleted position and the
coordinate degree for this reindexed cover in the characteristic-map
orientations, then
\begin{equation}
 \label{eq:height-coefficient}
 \mu(v,w)=\varepsilon_a\varepsilon_b(-1)^I\sigma\,
 \begin{cases}
 \eta,&\htc((\gamma')^\vee)\text{ even},\\
 0,&\htc((\gamma')^\vee)\text{ odd}.
 \end{cases}
\end{equation}
Thus the relevant geometric degree is the normal determinant degree.
The degree \(\htc(\gamma^\vee)+1\) occurring in the proof of
Lemma~\ref{lem:normal-determinant-degree} is instead the degree of
\(\det(TY|_C)\).

\begin{proposition}
\label{lem:partial-flag-orientations}
Let \(p:G/B\to G/P_\Theta\).  Choose the ordinary full and partial
cell orientations compatibly on minimal representatives, and use the
same \(\dot w_0\) to define the normal orientations.  There is a
surjective chain map
\[
 Q:\Ccell_*(G/B)\longrightarrow\Ccell_*(G/P_\Theta)
\]
whose action on a cell \(wu\), where \(w\in W^\Theta\) and
\(u\in W_\Theta\), is
\begin{equation}
 \label{eq:maximal-quotient}
 Q_i[wu]=
 \begin{cases}
 [w],&u=w_\Theta,\\
 0,&u\ne w_\Theta.
 \end{cases}
\end{equation}
In degree zero this is the identity on \(\Z\).  For any cover
\(w<v\) with \(\operatorname{codim}C_w\geq2\),
\[
 \mu_{G/P_\Theta}(v,w)
 =\mu_{G/B}(vw_\Theta,ww_\Theta).
\]
The nonmaximal summands form the kernel subcomplex, so the partial
cellular complex is its quotient.
\end{proposition}

\begin{proof}
Write \(N=\dim(G/B)=d+\ell(w_\Theta)\).  Length additivity gives
\[
 \operatorname{codim}_{G/B}C_{wu}
 =\operatorname{codim}_{G/P_\Theta}C_w
   +\ell(w_\Theta)-\ell(u).
\]
The projection therefore preserves the open filtrations.  A
codimension-\(i\) nonmaximal cell maps into \(\Omega_{i-1}\) on the
base, and hence induces zero on the degree-\(i\) filtration quotient.
For \(u=w_\Theta\), the inverse image of \(C_w\) has
\(C_{ww_\Theta}\) as its open cell.  Smooth base change gives
\[
 N_{C_{ww_\Theta}/(G/B)}
 \cong p^*N_{C_w/(G/P_\Theta)}.
\]
The base map between these two affine cells induces the identity on
the corresponding Thom homology summand with the pullback
orientation.  These are precisely our orientations: after left
multiplication by \(\dot w_0\), the opposite cells at
\(ww_\Theta\) and \(w\) become the ordinary full and partial
cells with the same minimal label \(w_0ww_\Theta\), whose chosen
orientations agree under projection.  The normal map is consequently
orientation preserving at the fixed point.  The fixed Thom
normalization is the same on the full and partial normal determinant
lines, since their ranks agree.  Thus the orientations agree
everywhere on the cell.

The maps of filtration quotients commute with connecting morphisms,
so they give the chain map \(Q\) with the stated formula.  Each base
summand is the image of its maximal lift, proving surjectivity.
Comparing the \([v]\)-components of
\(Q_{i-1}\partial_i[ww_\Theta]=\partial_iQ_i[ww_\Theta]\)
proves the coefficient equality.  The span of the maximal lifts
need not itself be a subcomplex; the asserted construction is the
quotient by the nonmaximal summands.
\end{proof}

\begin{remark}
\label{rem:sl3-partial-check}
For \(SL_3/P_{\{\alpha_1\}}\simeq\mathbb P^2\), consider
\(w=e\), \(v=s_2\), and \(\gamma=\alpha_2\).
Since \(2\rho_\Theta=\alpha_1\),
\[
 \deg\det E_C
 =\langle\rho-\alpha_1,\alpha_2^\vee\rangle-1=1.
\]
Thus the differential \(\KMW_2\to\KMW_1\) is \(\pm\eta\).
For the same minimal cover in \(SL_3/B\), the degree is instead
\(\htc(\alpha_2^\vee)-1=0\), and the component vanishes.
The corresponding maximal cover is
\(s_1<s_2s_1\); its right reflection has root
\(s_1\alpha_2=\alpha_1+\alpha_2\), of coroot height two, and it
has the required nonzero component.  This also agrees with the
ordinary reindexed cover \(s_1s_2\gtrdot s_2\).
\end{remark}

\begin{corollary}
\label{cor:partial-eta-comparison}
For every partial flag \(X_\Theta=G/P_\Theta\),
\[
 \partial_i^{\A^1}[w]=
 \sum_{v\gtrdot w}
 \frac{c_\Theta(\pi_\Theta(w),\pi_\Theta(v))}{2}\eta[v]
 \qquad(i\geq2),
\]
where \(\ell(w)=d-i\).  The same formula follows from the full
flag comparison by taking the quotient in
Proposition~\ref{lem:partial-flag-orientations}.
\end{corollary}

\begin{proof}
The direct assertion is Theorem~\ref{thm:eta-comparison}.
For the quotient description, note that the maximal lift
\(ww_\Theta\) is reindexed in the full flag by
\(w_0ww_\Theta=\pi_\Theta(w)\).  This is a minimal representative
in the ordinary dimension-indexed complex, where projection to the
partial flag retains exactly those summands.  The two projections
therefore give the same signed matrix.
\end{proof}

\begin{theorem}
\label{thm:chain-model}
For every \(X_\Theta=G/P_\Theta\), the oriented
cellular \(\A^1\)-chain complex is
\[
 0\longrightarrow\bigoplus_{\ell(w)=0}\KMW_d
 \xrightarrow{\eta E_d^\Theta}\cdots
 \xrightarrow{\eta E_2^\Theta}\bigoplus_{\ell(w)=d-1}\KMW_1
 \xrightarrow{0}\Z\longrightarrow0,
\]
with the evident truncation in dimensions zero and one.
\end{theorem}

\begin{proof}
Combine Lemma~\ref{lem:bruhat-a1-cellular} with
Theorem~\ref{thm:eta-comparison}.
\end{proof}

\section{Smith decomposition and applications}
\label{sec:smith}

Let \(m_i(q)\) denote the number of elementary Smith summands
\[
  0\longrightarrow \Z
  \xrightarrow{q}
  \Z\longrightarrow 0
\]
placed in degrees \(i+1\) and \(i\) in the integral chain complex
\[
  \left(
    \bigoplus_{\ell(u)=\bullet}\Z[u],
    D_\bullet^\Theta
  \right),
\]
or equivalently, after the length-reversing permutation \(P_\bullet\), in
\[
  \left(
    \bigoplus_{\ell(w)=d-\bullet}\Z[w],
    E_\bullet^\Theta
  \right).
\]
Since \(D_1^\Theta=0\), one has \(m_0(q)=0\) for every \(q\).
Let \(b_i\) be the rank of the free homology of \(D_\bullet^\Theta\).
Define
\[
  {}_{q\eta}\KMW_i
  =
  \ker\!\left(q\eta\colon \KMW_i\to \KMW_{i-1}\right)
\]
and
\[
  \KMW_i/q\eta
  =
  \coker\!\left(q\eta\colon \KMW_{i+1}\to \KMW_i\right).
\]
For \(q=1\), \(\KMW_i/\eta\simeq \KM_i\), the Milnor \(K\)-theory sheaf.

\begin{theorem}
\label{thm:smith-formula}
Let \(X_\Theta=G/P_\Theta\) be as in
Section~\ref{sec:bruhat-cells}.
For \(i\geq 1\), there is a noncanonical isomorphism
\[
  \begin{aligned}
  \Hcell_i(X_\Theta)
  &\cong
  (\KMW_i)^{b_i}
  \oplus
  \bigoplus_{q\geq 1}(\KMW_i/q\eta)^{m_i(q)}\\
  &\qquad{}\oplus
  \bigoplus_{q\geq 1}({}_{q\eta}\KMW_i)^{m_{i-1}(q)}.
  \end{aligned}
\]
Moreover \(\Hcell_0(X_\Theta)\cong \Z\).
\end{theorem}

\begin{proof}
The structure theorem for finite free chain complexes over the PID
\(\Z\) decomposes the complex \(D_\bullet^\Theta\), and
therefore also the reindexed complex \(E_\bullet^\Theta\),
up to integral chain isomorphism, into free homology generators and
two-term elementary complexes \(\Z\xrightarrow{q}\Z\).  Concretely, the
decomposition is
obtained by degreewise changes of basis by unimodular integer matrices.
The same matrices act as automorphisms of direct sums of Milnor--Witt
sheaves, and they commute with the natural operation of multiplying each
integer matrix entry by \(\eta\).  Hence the Smith decomposition of
\(E_\bullet^\Theta\) induces an isomorphic decomposition of the cellular $\mathbb{A}^1$-
chain complex \(\eta E_\bullet^\Theta\) in the abelian category of
strictly \(\A^1\)-invariant sheaves.  No division by an elementary divisor
is involved.

After applying Theorem~\ref{thm:chain-model}, each elementary summand
\(\Z\xrightarrow{q}\Z\) placed in degrees \(i+1\) and \(i\), with
\(i\geq1\), becomes
\[
  0\longrightarrow \KMW_{i+1}
  \xrightarrow{q\eta}
  \KMW_i\longrightarrow 0.
\]
The bottom differential is the separate zero map
\[
  \bigoplus_{\ell(w)=d-1}\KMW_1\longrightarrow \Z,
\]
and \(m_0(q)=0\) records that there is no non-zero elementary Smith block
in degrees \(1\) and \(0\).  The homology of such a two-term summand is
\(\KMW_i/q\eta\) in degree \(i\) and
\({}_{q\eta}\KMW_{i+1}\) in degree \(i+1\).  Therefore the kernel term in
\(\Hcell_i\) comes from the elementary summands counted by
\(m_{i-1}(q)\), namely the summands placed in degrees \(i\) and \(i-1\).
Summing these contributions with the free summands gives the formula.
The dependence on the base field is contained in the kernel and
cokernel sheaves of multiplication by \(q\eta\).
\end{proof}

The Smith formula is already a statement over the ground field \(k\): all
base-field dependence is contained in the Milnor--Witt sheaves, while the
matrices \(E_i^\Theta\) are integral matrices.  Since
\(E_\bullet^\Theta\) is obtained from \(D_\bullet^\Theta\) by degreewise
permutation matrices, it has the same Smith data in the same degrees.  A
common simplification occurs when the dimension-indexed complex
\((A_\bullet^\Theta,D_\bullet^\Theta)\) has only unit elementary
divisors.
In that case let
\[
  \beta_j=\rank H_j(A_\bullet^\Theta,D_\bullet^\Theta),
\]
and let \(T_j\) be the number of Smith summands
\[
  \Z\xrightarrow{1}\Z
\]
placed in dimensions \(j+1\) and \(j\).
Set \(T_j=0\) outside \(\{0,\ldots,d\}\).  When \(k=\R\), the real cellular
boundary matrices are \(\delta^\Theta=2D^\Theta\), and this condition is
equivalent to
\[
  H_j(X_{\Theta,\R}(\R),\Z)\cong \Z^{\beta_j}\oplus(\Z/2)^{T_j}.
\]
Thus the real Betti and \(2\)-torsion ranks are a convenient way to record
the same integer Smith data, but they are not needed to perform the
motivic computation over a general field.

One should not read the resulting direct sum as saying that the
\(\A^1\)-homology is torsion-free.  The free summands are only the
\(\KMW_i\)-summands.  Each unit elementary summand of \(D_\bullet^\Theta\)
is refined motivically by the two homology sheaves of
the elementary cone
\[
  0\longrightarrow \KMW_{i+1}
  \xrightarrow{\eta}
  \KMW_i\longrightarrow 0,
\]
namely \(\KMW_i/\eta\simeq\KM_i\) and
\({}_{\eta}\KMW_{i+1}\).  Thus the \(K^M\)-terms and the
\(\eta\)-kernel terms are the motivic summands corresponding to the
topological \(\Z/2\)-torsion Smith blocks; they should not be confused
with ordinary torsion sheaves.

\begin{corollary}
\label{cor:two-torsion}
Let \(X_\Theta=G/P_\Theta\) be as in
Section~\ref{sec:bruhat-cells}.
Assume every non-zero elementary divisor of
\((A_\bullet^\Theta,D_\bullet^\Theta)\) is \(1\), and let
\(\beta_j,T_j\) be the associated integer Smith data defined above.  Then,
for \(i\geq 1\),
\[
  \Hcell_i(X_\Theta)
  \cong
  (\KMW_i)^{\beta_i}
  \oplus
  (\KM_i)^{T_i}
  \oplus
  ({}_{\eta}\KMW_i)^{T_{i-1}},
\]
and \(\Hcell_0(X_\Theta)\cong\Z\).
\end{corollary}

\begin{proof}
The complex \(E_\bullet^\Theta\) is obtained from \(D_\bullet^\Theta\)
through the degree-preserving permutation
\(P_\bullet\).  Thus its free homology in degree \(i\) has rank
\(\beta_i\), and a unit elementary summand
\(\Z\xrightarrow{1}\Z\) in dimensions \(i+1,i\) contributes the cokernel
term \(\KMW_i/\eta\simeq\KM_i\) in cellular degree \(i\).  The unit
summands in dimensions \(i,i-1\), counted by \(T_{i-1}\), contribute the
kernel term \({}_{\eta}\KMW_i\).  This is exactly
Theorem~\ref{thm:smith-formula} with all elementary divisors equal to
\(1\).
\end{proof}

\begin{corollary}
\label{cor:lr-tables}
Let \(X_\Theta=G/P_\Theta\) be a split semisimple simply connected partial
flag variety over a perfect field \(k\) with
\(\operatorname{char}k\neq2\) whose root datum and parabolic type are
among the available computations in classical type \(B_n,C_n,D_n\) with
\(n\leq 7\), or in exceptional type \(F_4,E_6,E_7\).  Let
\(E_\bullet^\Theta=P_{\bullet-1}^{-1}D_\bullet^\Theta P_\bullet\) be the
reindexed matrices obtained from the computation of signed boundary
matrices, and let \(b_i,m_i(q)\) be the Smith data of this integer
complex.  Then the cellular
\(\A^1\)-homology over \(k\) is given by the Smith formula of
Theorem~\ref{thm:smith-formula}.

If, moreover, the computed Smith decomposition of
\(D_\bullet^\Theta\) has only unit non-zero elementary divisors, or
equivalently the split real specialization has integral homology
\[
  H_j(X_{\Theta,\R}(\R),\Z)\cong \Z^{\beta_j}\oplus(\Z/2)^{T_j},
\]
then
\[
  \Hcell_i(X_\Theta)
  \cong
  (\KMW_i)^{\beta_i}
  \oplus
  (\KM_i)^{T_i}
  \oplus
  ({}_{\eta}\KMW_i)^{T_{i-1}}
  \qquad (i\geq 1),
\]
with \(\Hcell_0(X_\Theta)\cong\Z\).  Equivalently, every
available matrix \(D_i^\Theta\) has a cellular \(\A^1\)-homology lift
over \(k\); when all non-zero Smith elementary divisors are \(1\), this lift is obtained by replacing
each elementary summand \(\Z\xrightarrow{1}\Z\) of \(D_\bullet^\Theta\) by
\(\KMW_{i+1}\xrightarrow{\eta}\KMW_i\).
\end{corollary}

\begin{proof}
Theorem~\ref{thm:eta-comparison} for the full flag and
Corollary~\ref{cor:partial-eta-comparison} for projected partial flags
represent the cellular $\mathbb{A}^1$-differential by the length-reversing
reindexing of \(D_i^\Theta\), multiplied by \(\eta\).  Therefore the
Smith data of the integer matrices enter directly into
Theorem~\ref{thm:smith-formula}.  The equivalence with ordinary
\(2\)-torsion is the elementary Smith comparison for the real cellular
complex: the real cellular boundary matrices are \(2D_\bullet^\Theta\), so an
elementary summand
\(\Z\xrightarrow{q}\Z\) contributes \(\Z/(2q)\) to the corresponding
ordinary cellular homology summand.  Thus the ordinary torsion is a sum of
\(\Z/2\)'s exactly when every non-zero Smith elementary divisor of \(D_\bullet^\Theta\) is
\(q=1\).  Under this condition,
Corollary~\ref{cor:two-torsion} gives the displayed decomposition.  The
last sentence is the same statement on each elementary Smith summand.
\end{proof}

\begin{remark}
\label{rem:type-a-wendt}
In type \(A\), the signed coefficient formula supplies the integer
matrices \(D_i^\Theta\), and Hudson--Matszangosz--Wendt prove via
Witt-sheaf cohomology and Chow--Witt rings of partial flag varieties that
the integral cohomology torsion of the split real specialization is
\(2\)-torsion.  Since \(X_{\Theta,\R}(\R)\) is a finite CW complex, the
universal coefficient theorem transfers this torsion statement between
integral cohomology and integral homology.  Thus the type \(A\) case falls
under Corollary~\ref{cor:two-torsion}.  After specialization to \(\R\),
the result is compatible with the Hudson--Matszangosz--Wendt ring-level
calculations: their work supplies the topological \(2\)-torsion input,
while the present theorem gives the corresponding Milnor--Witt cellular
chain complex and homology sheaves over every perfect base field
\(k\) with \(\operatorname{char}k\neq2\).

For complete flags in type \(A\) over infinite perfect fields of
characteristic different from two, Yang's decomposition in the stable
category of Milnor--Witt motives gives sums of Tate motives and
twists of the cone of \(\eta\)
\cite[Theorem~5.8]{YangSplit}.  Its real cohomological consequence is
that all torsion is killed by \(2\)
\cite[Remark~2.19]{YangSplit}.  The Smith decomposition above takes
place in the cellular complex of strictly \(\A^1\)-invariant sheaves;
its kernel and cokernel summands describe cellular homology.

There is also a direct additive comparison with Chow--Witt groups.
Morel--Sawant prove that, for a cellular smooth scheme \(X\) and any
strictly \(\A^1\)-invariant sheaf \(M\), Nisnevich cohomology is computed
by the cellular cochain complex \cite[Proposition~2.27]{MorelSawant}
\[
  C^n_{\cell}(X;M)
  =
  \Hom_{\mathbf{Ab}_{\A^1}(k)}(\Ccell_n(X),M).
\]
Taking \(M=\KMW_p\) gives the untwisted Chow--Witt group
\[
  \widetilde{CH}^{p}(X_\Theta)
  =
  H^p_{\mathrm{Nis}}(X_\Theta,\KMW_p)
  \cong
  H^p\Hom(\Ccell_*(X_\Theta),\KMW_p).
\]
For \(n\geq1\),
\(\Hom(\KMW_n,\KMW_p)\cong \KMW_{p-n}(k)\).  Thus, for \(p\geq1\),
the Smith decomposition of \(E_\bullet^\Theta\) gives the additive group formula
\[
  \widetilde{CH}^{p}(X_\Theta)
  \cong
  \GW(k)^{b_p}
  \oplus
  \bigoplus_{q\geq1}
  \left(\KMW_0(k)/q\eta\KMW_1(k)\right)^{m_{p-1}(q)}
  \oplus
  \bigoplus_{q\geq1}
  \left({}_{q\eta}\KMW_0(k)\right)^{m_p(q)} ,
\]
where
\[
  {}_{q\eta}\KMW_0(k)
  =
  \ker\!\left(q\eta\colon \KMW_0(k)\to\KMW_{-1}(k)\right).
\]
When all non-zero Smith elementary divisors are \(1\), this becomes,
for \(p\geq1\),
\[
  \widetilde{CH}^{p}(X_\Theta)
  \cong
  \GW(k)^{\beta_p}
  \oplus
  \Z^{T_{p-1}}
  \oplus
  ({}_{\eta}\GW(k))^{T_p} .
\]
Here \(\KMW_0(k)=\GW(k)\), \(\KMW_0(k)/\eta\KMW_1(k)\cong\Z\), and
\({}_{\eta}\GW(k)=\ker(\GW(k)\to W(k))=\Z\cdot(\langle1\rangle+
\langle-1\rangle)\).  In degree zero, \(\Ccell_0(X_\Theta)=\Z\)
and \(\partial_1=0\) give separately
\[
 \widetilde{CH}^{0}(X_\Theta)
 \cong\Hom(\Z,\KMW_0)=\GW(k).
\]
Thus, at the level of additive groups, the Chow--Witt computation has the same Smith-block origin as
Corollary~\ref{cor:two-torsion}; the ring structure requires the
additional cup-product data computed in the Chow--Witt setting and is not
determined by the Smith normal form alone.
\end{remark}

The torsion ranks can be recovered from cell numbers and Betti numbers.
Let
\[
  c_i^\Theta=\#\{w\in W^\Theta:\length(w)=i\}.
\]
When only \(2\)-torsion occurs, the ranks satisfy
\[
  T_{-1}=0,\qquad
  T_i=c_i^\Theta-\beta_i-T_{i-1}.
\]
Indeed \(c_i^\Theta=\beta_i+\rank D_i+\rank D_{i+1}\), while
\(T_{i-1}=\rank D_i\) and \(T_i=\rank D_{i+1}\).

\subsection{Comparison with Morel--Sawant for full flags}
\label{sec:morel-sawant-comparison}

Morel--Sawant construct the Bruhat cellular complex of \(G/B\) in
all degrees and compute its differentials in degrees one and two
\cite[\S5.3]{MorelSawant}.  For \(G\) split, semisimple, almost
simple, and simply connected, their Theorem~5.8 also identifies
\[
 \Ccell_*(G/B)\cong\Ccell_*(G/T)
 \cong\Ccell_*(G)\otimes_{\mathbf Z_{\A^1}[T]}\Z.
\]
We use their cellular construction and compare its degree-two
differential with the formula for the real cellular boundary matrix explicitly.

Write \(e_i\) and \(e_{ij}\) for the generators labelled by
\(w_0s_i\) and \(w_0s_is_j\), respectively, with \(i\ne j\).
If \(s_i\) and \(s_j\) commute, \(e_{ij}\) and \(e_{ji}\) refer
to the same cell.  Set
\[
 n_{ji}=-\langle\alpha_j,\alpha_i^\vee\rangle,
 \qquad
 s_j(\alpha_i^\vee)=\alpha_i^\vee+n_{ji}\alpha_j^\vee.
\]
This is the Cartan convention preceding
\cite[(4.20)]{MorelSawant}.\footnote{The final identification of
\(n_{ji}\) in the printed statement of
\cite[Theorem~5.8(b)]{MorelSawant} differs from the convention
preceding (4.20).  We use the latter, also repeated in their proof
of Theorem~5.1; the displayed reflection identity fixes the order
of the root and coroot.}
In the normal-word orientations of
\cite[Theorem~5.8]{MorelSawant}, the component of
\(\overline\partial_2(e_{ij})\) along \(e_i\) is zero, and its
component along \(e_j\) is \((n_{ji})_\epsilon\eta\), where
\(n_\epsilon=\sum_{a=0}^{n-1}\langle(-1)^a\rangle\) for
\(n\geq1\) and \(0_\epsilon=0\).  In
particular, this component is zero for even \(n_{ji}\) and is
\(\eta\) for odd \(n_{ji}\).

The root-curve calculation in Section~\ref{sec:comparison} gives
the same parity rule.  For the cover
\(w_0s_is_j\to w_0s_i\), the reflected root is \(\alpha_j\),
so the determinant degree is zero.  For the cover
\(w_0s_is_j\to w_0s_j\), the reflected coroot is
\(s_j(\alpha_i^\vee)\), of height \(1+n_{ji}\), and the
determinant degree of Lemma~\ref{lem:normal-determinant-degree} is
\(n_{ji}\).  The integral orientations of
Section~\ref{sec:bruhat-cells} can change the signs of the nonzero
components; their simultaneous values are specified by
Theorem~\ref{thm:eta-comparison}.  Their images are unchanged.

For completeness, put
\[
 J=\{j:\ n_{ji}\text{ is odd for some }i\ne j\}.
\]
Each nonzero degree-two column has a single \(\eta\)-entry, up
to sign.  Hence
\[
 \im\partial_2=\bigoplus_{j\in J}\eta\KMW_2\,e_j,
 \qquad
 \Hcell_1(G/B)\cong
 (\KM_1)^{|J|}\oplus(\KMW_1)^{r-|J|}.
\]
For an irreducible root system, every vertex belongs to \(J\)
outside symplectic type.  In symplectic type, \(J\) consists of
all vertices except the unique long simple root.  This recovers
precisely \cite[Corollary~5.9]{MorelSawant}:
\begin{equation}
\label{eq:morel-sawant-h1}
 \Hcell_1(G/B)\cong
 \begin{cases}
  (\KM_1)^r,&G\text{ not of symplectic type},\\
  (\KM_1)^{r-1}\oplus\KMW_1,&G\text{ of symplectic type}.
 \end{cases}
\end{equation}
Here \(\KM_1\cong\Gm\), and symplectic type includes
\(A_1=C_1\) and \(B_2=C_2\).  In rank one this gives
\(\Hcell_1(\mathbb P^1)=\KMW_1\).
The same calculation shows that the nonzero elementary divisors
of \(E_2\) are all one,
with \(m_1(1)=|J|\), \(b_1=r-|J|\), and \(m_0(q)=0\) for
every \(q\).
Thus the degree-one cases for \(SL_3/B\), \(F_4/B\), and
\(G_2/B\) below agree with Morel--Sawant's result.

Their low-degree formula holds in every root type.
Their Remark~5.10 leaves the complete
full-flag homology calculation to subsequent work.
Theorem~\ref{thm:eta-comparison} determines all differentials for
every root type over perfect fields of
characteristic different from two.  Applying integral Smith
decomposition, Theorem~\ref{thm:smith-formula} gives the homology
sheaves in every degree in terms of the integer Bruhat matrices.
This extends their full-flag calculation to higher degrees;
Proposition~\ref{lem:partial-flag-orientations} and
Corollary~\ref{cor:partial-eta-comparison} also give the corresponding
calculation for partial flags.
Projective spaces were computed in all degrees in
\cite[Corollary~2.51]{MorelSawant}.
The higher-degree groups throughout this paper are cellular
\(\A^1\)-homology sheaves: the general comparison with ordinary
\(\A^1\)-homology gives an isomorphism in degree one and an
epimorphism onto cellular homology in degree two for an
\(\A^1\)-connected cellular smooth scheme
\cite[Corollary~2.30]{MorelSawant}.

\subsection{Second homology of full flags in type \texorpdfstring{$A$}{A}}
\label{sec:type-a-second-homology}

For full flags in type \(A\), the Smith formula gives a uniform
degree-two computation without constructing the individual boundary
matrices.

\begin{corollary}
\label{cor:type-a-second-homology}
Let \(k\) be a perfect field of characteristic different from two,
let \(n\geq3\), and let \(B\subset SL_n\) be a split Borel subgroup.
There is a noncanonical isomorphism
\[
 \Hcell_2(SL_n/B)
 \cong
 (\KM_2)^{n(n-3)/2}
 \oplus
 ({}_{\eta}\KMW_2)^{n-1}.
\]
\end{corollary}

\begin{proof}
The Bruhat cell numbers are the coefficients of
\[
 \sum_i c_i t^i=\prod_{j=2}^{n}(1+t+\cdots+t^{j-1}),
\]
so
\[
 c_1=n-1,\qquad
 c_2=\binom{n-1}{2}+n-2=\frac{(n-2)(n+1)}2.
\]
The Betti polynomials in \cite[Theorem~9.6]{HudsonMatszangoszWendt}
give \(\beta_1=\beta_2=0\) for the real full flag manifold.  Its
integral cohomology torsion is annihilated by two by
\cite[Theorem~5.8 and Remark~2.19]{YangSplit}; see also
\cite[Theorem~9.5]{HudsonMatszangoszWendt}.  The universal
coefficient theorem gives the same torsion assertion for integral
homology.  Since \(T_0=0\), the identity
\(c_i=\beta_i+T_i+T_{i-1}\) gives
\enlargethispage{\baselineskip}
\[
 T_1=n-1,\qquad
 T_2=c_2-T_1=\frac{n(n-3)}2.
\]
The assertion follows from Corollary~\ref{cor:two-torsion}.
\end{proof}

\subsection{Algorithmic form}

For a fixed split Dynkin type and subset \(\Theta\), the
signed real cellular boundary matrices determine the motivic calculation
as follows.

\begin{enumerate}[label=\textbf{Step \arabic*.},wide=0pt,leftmargin=*]
  \item Fix Bourbaki numbering, a split pinning, and reduced-word
        parametrizations for the ordinary cells.  Record the orientation
        signs \(\varepsilon_w\) of Proposition~\ref{prop:lr-boundary}.  Use
        Section~\ref{sec:bruhat-cells} to orient the positive cells
        normally by means of the opposite cells.
  \item Enumerate the length-graded set \(W^\Theta\) and all Bruhat covers
        \(w\gtrdot w'\).
  \item For each ordinary cover \(a\gtrdot b\), compute the data for the boundary coefficient:
        the deleted position \(I(a,b)\), the root \(\gamma(a,b)\), and the
        degree of the real coordinate change for the chosen reduced words.
        For a full flag, if \(\sigma(a,b)\) is this integer degree, form
        \[
          c(a,b)=\varepsilon_a\varepsilon_b
          (-1)^{I(a,b)}\sigma(a,b)
          \bigl(1+(-1)^{\htc(\gamma(a,b)^\vee)}\bigr).
        \]
        For a partial flag, restrict these coefficients to the minimal
        coset representatives as in Proposition~\ref{prop:lr-boundary}.
        The comparison theorem identifies the corresponding motivic
        entry with \((c_\Theta(a,b)/2)\eta\).
  \item Form \(D_i^\Theta=(c_\Theta(w,w')/2)\).  Check
        \(D_{i-1}^\Theta D_i^\Theta=0\).
  \item Form the length-reversing permutation
        \(\pi_\Theta(w)=w_0ww_\Theta\) and the matrices
        \(P_i[w]=[\pi_\Theta(w)]\).  Then form the
        codimension-indexed matrices
        \[
          E_i^\Theta=P_{i-1}^{-1}D_i^\Theta P_i.
        \]
        Equivalently, the coefficient of \([v]\) in \(E_i^\Theta[w]\)
        is \(c_\Theta(\pi_\Theta(w),\pi_\Theta(v))/2\).  Compute a Smith
        decomposition of \(E_\bullet^\Theta\), and insert the elementary
        divisors into Theorem~\ref{thm:smith-formula}.
  \item If the Smith data of \(D_\bullet^\Theta\) are already known to have only unit
        elementary divisors, use the shorter formula in
        Corollary~\ref{cor:two-torsion}.  For \(k=\R\), this is
        equivalently the statement that the real cellular homology is
        \(\Z^{\beta_i}\oplus(\Z/2)^{T_i}\).
\end{enumerate}

The cited work provides formulas for signed boundary coefficients in type \(A\),
together with computations of boundary matrices for classical types \(B_n,C_n,D_n\) with
\(n\leq 7\) and for exceptional types \(F_4,E_6,E_7\)
\cite[Sections~6.3--6.4]{LambertRabelo2026}.  The accompanying Sage
implementation constructs the boundary matrices and computes the Smith
data \cite{LambertCode}.  These matrices and Smith data are external
integer inputs to the present paper; the contribution here is the
Milnor--Witt lift of those inputs via
Theorem~\ref{thm:eta-comparison} and
Corollary~\ref{cor:partial-eta-comparison}.  Thus each of these computations of integer
boundary matrices gives a finite computation of cellular
\(\A^1\)-homology sheaves over \(k\).  In the cases covered by
the cited tables, the Poincar\'e polynomials determine the free ranks
\(\beta_i\);
together with the cell numbers
\(c_i^\Theta\), the recurrence
\[
  T_{-1}=0,\qquad T_i=c_i^\Theta-\beta_i-T_{i-1}
\]
determines the \(2\)-torsion ranks in the elementary \(2\)-torsion cases.
Corollary~\ref{cor:lr-tables} then gives the corresponding sheaf decomposition.  Thus the topological applications in the
cited tables, including the Poincar\'e-polynomial and
orientability computations, can be used as integer inputs for the
motivic calculation without recomputing the \(\A^1\)-local degrees
separately.

\subsection{The torus-support spectral sequence}
\label{sec:torus-support}

Throughout this subsection, \(k\) has characteristic zero and \(G\)
is any split semisimple simply connected group.  There is no
restriction on its root type or rank.  Write \(B=T\ltimes U\),
where \(T\) has rank \(r\).  As an application to our subsequent
work \cite{LiuLiuTorus}, we apply Smith decomposition to its
explicit torus-enriched operators and compute the homology sheaves
of the augmented and associated graded complexes.
Put
\[
 \mathcal C_*=\Ccell_*(G/U),\qquad
 \mathbf H_T=\mathbf Z_{\A^1}[T],\qquad
 \epsilon_T:\mathbf H_T\longrightarrow\Z.
\]
Tensor products in this subsection are taken in the category of
strictly \(\A^1\)-invariant sheaves.  The augmentation description
is due to \cite[Theorem~5.8]{MorelSawant}, as recalled in
Section~\ref{sec:morel-sawant-comparison}.  In the \(G/U\) model,
\cite[Theorem~7.1]{LiuLiuTorus} identifies chain complexes
\begin{equation}
\label{eq:torus-augmentation}
 \mathcal C_*\otimes_{\mathbf H_T}\Z
 \cong\Ccell_*(G/B).
\end{equation}
The cover operators in \cite[Theorem~7.1]{LiuLiuTorus} provide the
integer matrices for the Smith argument directly, as we verify below.

The simple coroots identify \(T\cong\Gm^r\) and give
\[
 \mathbf H_T\cong\bigoplus_{S\subseteq\{1,\ldots,r\}}H_S,
 \qquad H_\varnothing=\Z,\qquad
 H_S=\KMW_{|S|}\quad(S\ne\varnothing).
\]
The ideals \(F^q\mathbf H_T=\bigoplus_{|S|\geq q}H_S\) induce the
finite torus-support filtration of \(\mathcal C_*\).  By
\cite[Proposition~10.1]{LiuLiuTorus}, its spectral sequence has
\begin{equation}
\label{eq:torus-support-ss}
 \mathcal E^1_{q,p-q}
 \cong\bigoplus_{|S|=q}H_p(\mathcal C_{*,S})
 \Longrightarrow\Hcell_p(G/U),
\end{equation}
where \(\mathcal C_{*,S}\) is the summand with support exactly \(S\)
in the associated graded.  In particular,
\(\mathcal E^1_{0,p}\cong\Hcell_p(G/B)\).  For \(q=|S|>0\),
the terms are
\[
 \mathcal C_{p,S}\cong
 \bigoplus_{\ell(a)=p}\KMW_{p+q}\quad(p\geq0).
\]
At \(p=0\), this means \(\mathcal C_{0,S}=H_S\); the degree-zero
term before taking support is \(\mathbf H_T\).

Fix the normal words and pinned frames of \cite{LiuLiuTorus}.
The frame degrees lie in \(\{\langle1\rangle,\langle-1\rangle\}\)
by \cite[Proposition~2.6]{LiuLiuTorus}.
For a reduced deletion in position \(h\) of a word for \(a\), with
target \(b\), write its data as
\[
 \lambda_{a,h}=\sum_{j=1}^r n_j(a,h)\alpha_j^\vee,
 \qquad \sigma_{a,h}\in\Z,
 \qquad \vartheta_{a,h}=\langle j_{a,h}\rangle,
 \quad j_{a,h}\in\{1,-1\}.
\]
Here \(\lambda\) is the transported coroot, \(\sigma\) is the
normal determinant exponent, and \(\vartheta\) is the degree of
the change to the target normal word.  Define integer matrices
\begin{equation}
\label{eq:support-integer-matrices}
 A^S_p(b,a)=
 \sum_{h:\,a\setminus h=b}
 (-1)^{h-1}j_{a,h}\,
 \frac{1-(-1)^{\sigma_{a,h}+\sum_{j\in S}n_j(a,h)}}{2},
 \qquad \ell(a)=p.
\end{equation}
The sum is over reduced deletions yielding \(b\).  These matrices
retain the coroot coordinates as well as the flag boundary coefficients.

\begin{proposition}
\label{prop:support-smith}
For every \(S\), the matrices \(A^S_*\) form an integer chain
complex.  For nonempty \(S\), the differential of
\(\mathcal C_{*,S}\) in the preceding bases is \(\eta A^S_*\).
For \(S=\varnothing\), this holds in degrees \(p\geq2\), and
\(\partial_1=A_1^\varnothing=0\).  Let \(b_p^S\) be the
free homology ranks, and let \(m_p^S(d)\) count its Smith summands
\(\Z\xrightarrow{d}\Z\) in degrees \(p+1,p\), for \(d\geq1\).
If \(q=|S|>0\), then for every \(p\geq0\) there is a
noncanonical isomorphism
\begin{equation}
\label{eq:support-smith}
\begin{aligned}
 H_p(\mathcal C_{*,S})
 &\cong(\KMW_{p+q})^{b_p^S}
 \oplus\bigoplus_{d\geq1}
       (\KMW_{p+q}/d\eta)^{m_p^S(d)}\\
 &\qquad\oplus\bigoplus_{d\geq1}
       ({}_{d\eta}\KMW_{p+q})^{m_{p-1}^S(d)},
 \qquad m_{-1}^S(d)=0.
\end{aligned}
\end{equation}
For \(S=\varnothing\), the same formula holds with \(q=0\) for
\(p\geq1\), while \(H_0(\mathcal C_{*,\varnothing})=\Z\)
and \(m_0^\varnothing(d)=0\).  These formulas determine the entire
\(\mathcal E^1\)-page of \eqref{eq:torus-support-ss} from the
integer Smith data.
\end{proposition}

\begin{proof}
In one torus coordinate, the Milnor--Witt product identity gives
\[
 [a]([v]-[1])=[av]-[a]
 =\langle a\rangle([v]-[1]).
\]
Consequently \([\lambda_S(u)]\) acts on \(H_S\) by
\(\langle u^{\sum_{j\in S}n_j}\rangle\).  The cover operator of
\cite[Proposition~10.1]{LiuLiuTorus} therefore sends
\[
 (u)\otimes x\otimes h\longmapsto
 \bigl(\langle u^{\sigma+\sum_{j\in S}n_j}\rangle-1\bigr)xh.
\]
An even exponent gives zero; an odd exponent gives
\(\eta(u)xh\).  The localized tensor identifications of
\cite[Lemma~2.1]{LiuLiuTorus} identify the source and target with
\(\KMW_{p+q}\) and \(\KMW_{p+q-1}\).  The symmetries in the
ordered cut contribute powers of
\(\varepsilon=-\langle-1\rangle\), which disappear after
multiplication by \(\eta\), since \(\eta\varepsilon=\eta\).
Also \(\vartheta_{a,h}\eta=j_{a,h}\eta\).  Including the ordinary
face sign gives precisely \eqref{eq:support-integer-matrices}.
For \(p=1\), the determinant exponent is zero and the operator is
\((u)\otimes h\mapsto([\lambda_S(u)]-1)h\); the same calculation
applies directly for nonempty \(S\), without introducing a
\(\KMW_0\)-factor.  For empty \(S\) and \(p\geq2\), set
\(h=1\) in the preceding calculation.  For empty \(S\) and
\(p=1\), torus augmentation kills \([\lambda(u)]-[1]\), and
the corresponding integer entry is zero since \(\sigma=0\).

In the real associated graded of the same filtered construction,
\cite[Corollary~10.2]{LiuLiuTorus} gives the integer boundary matrices
\(\delta^S_*=-2A^S_*\).  Since these are complexes of free abelian
groups, \((\delta^S)^2=0\) implies \((A^S)^2=0\).
This uses the integral complex of the realized filtration.
Finally, the unimodular basis changes in a Smith decomposition of
\(A^S_*\) act on the corresponding Milnor--Witt sheaves.  The
argument in the proof of Theorem~\ref{thm:smith-formula}, with
degree \(p\) replaced in the sheaf index by \(p+q\), gives
\eqref{eq:support-smith}.  In empty support the bottom differential
is zero and the degree-zero term is \(\Z\).  In positive support
the differential in degrees \(1,0\) can be nonzero, so
\(m_0^S(d)\) is retained.
\end{proof}

\begin{corollary}
\label{cor:support-bottom-row}
For \(1\leq q\leq r\),
\[
 \mathcal E^1_{q,-q}\cong(\KM_q)^{\binom rq}.
\]
In the simple-coroot coordinates, the first support differential
\[
 d^1:\mathcal E^1_{0,1}=\Hcell_1(G/B)
 \longrightarrow\mathcal E^1_{1,-1}=(\KM_1)^r
\]
is induced by the coordinatewise quotient
\((\KMW_1)^r\to(\KM_1)^r\).
If \(G\) is almost simple, this is an isomorphism outside
symplectic type.  In symplectic type, under
\eqref{eq:morel-sawant-h1}, it is the identity on
\((\KM_1)^{r-1}\) together with the quotient
\(\KMW_1\to\KM_1\).
More generally, for \(1\leq q\leq r\),
\[
 d^1:\mathcal E^1_{q-1,2-q}\longrightarrow
      \mathcal E^1_{q,-q}\quad\text{is surjective},
 \qquad \mathcal E^2_{q,-q}=0.
\]
\end{corollary}

\begin{proof}
For the degree-one cell \(s_i\), the cover data are
\(\lambda=\alpha_i^\vee\), \(\sigma=0\), and
\(\vartheta=\langle1\rangle\).  Thus \(A_1^S\) is the row
whose \(i\)-th entry is one if \(i\in S\) and zero otherwise.
For nonempty \(S\), the image of the degree-one differential is
\(\eta\KMW_{q+1}\subset\KMW_q\), giving
\(H_0(\mathcal C_{*,S})=\KM_q\).  There are \(\binom rq\)
supports of size \(q\).

The part of the degree-one boundary map that raises support from zero
to one sends \((u)e_i\) to
\([\alpha_i^\vee(u)]-[1]\).  Modulo the image \(\eta\KMW_2\) of the
support-preserving differential, this is the quotient to the \(i\)-th
\(\KM_1\)-summand.  Every component of the degree-two flag boundary
map is an \(\eta\)-multiple, so this map descends to \(\Hcell_1(G/B)\).
Formula~\eqref{eq:morel-sawant-h1} gives the assertions for
almost simple groups.

For the general surjectivity, fix a support \(S'\) of size \(q\),
choose its least index \(i\), and put \(S=S'\setminus\{i\}\).
The degree-one \(e_i\)-summand with support \(S\) consists of
cycles for the support-preserving differential because \(i\notin S\).
The boundary map raises support to \(S'\) and induces the quotient
\(\KMW_q\to\KM_q\) on the indicated target summand, with no
component on another support of size \(q\).  Thus \(d^1\) is
surjective on each summand.  The outgoing differential from total
degree zero has zero target, proving the assertion about
\(\mathcal E^2\).
\end{proof}

Under \((\KM_1)^r\cong T\), this recovers the isomorphism and
symplectic quotient occurring in the map to \(T\) discussed in
\cite[Remark~5.11]{MorelSawant}.  Proposition~\ref{prop:support-smith}
expresses the associated graded homology
\(H_p(\mathcal C_{*,S})\) in the spectral sequence of
\cite[Proposition~10.1]{LiuLiuTorus}.

The higher differentials in \eqref{eq:torus-support-ss} contain
additional information.  For example, for \(G=SL_2\), the rank-one
calculation of \cite[\S8.1]{LiuLiuTorus}, with the summands ordered
by support, reads
\[
 \KMW_1\oplus\KMW_2\longrightarrow\Z\oplus\KMW_1,
 \qquad (x,y)\longmapsto(0,x+\eta y).
\]
The two associated graded complexes are
\(\KMW_1\xrightarrow{0}\Z\) and
\(\KMW_2\xrightarrow{\eta}\KMW_1\).  The first spectral sequence
differential from support zero to support one is the quotient map
\[
 \mathcal E^1_{0,1}=\KMW_1
 \longrightarrow\mathcal E^1_{1,-1}=\KM_1.
\]
The total degree-one homology is instead \(\KMW_2\), represented
by \(y\mapsto(-\eta y,y)\).  Thus computing the abutment requires
the maps between distinct supports in addition to their Smith data.

\section{Example: \texorpdfstring{$SL_3/B$}{SL3/B}}
\label{sec:sl3-example}

Let \(G=SL_3\), let \(B\) be the standard Borel subgroup, and write the
simple reflections as \(s_1,s_2\), with simple roots
\(\alpha_1,\alpha_2\).  The positive non-simple root is
\(\alpha_{12}=\alpha_1+\alpha_2\), and
\(\htc(\alpha_1^\vee)=\htc(\alpha_2^\vee)=1\), while
\(\htc(\alpha_{12}^\vee)=2\).  Choose the horizontal words
\[
  s_1,\quad s_2,\quad s_1s_2,\quad s_2s_1,\quad
  w_0=s_1s_2s_1 .
\]
The cellular differential is expressed through the reindexing
\(\pi(w)=w_0w\).  A component \(w\to v\) of the cellular boundary map, with
\(v\gtrdot w\), is evaluated by the ordinary cover
\(\pi(w)\gtrdot\pi(v)\).  Take the ordinary cell orientations
with deletion sign \((-1)^I\), and the induced opposite-cell normal
orientations of Section~\ref{sec:bruhat-cells}.  The relevant data
are as follows.  The sign in the two
non-zero rows changes if one changes the corresponding cell generator; the
resulting chain complex is isomorphic.
\[
\begin{array}{c|c|c|c}
\text{cellular }w\to v & \text{ordinary cover }\pi(w)\gtrdot\pi(v)
& c(\pi(w),\pi(v))/2 & \mu(v,w)\\
\hline
e\to s_1 & w_0\gtrdot s_1s_2 & 0 & 0\\
e\to s_2 & w_0\gtrdot s_2s_1 & 0 & 0\\
s_1\to s_1s_2 & s_1s_2\gtrdot s_1 & 0 & 0\\
s_1\to s_2s_1 & s_1s_2\gtrdot s_2 & -1 & -\eta\\
s_2\to s_1s_2 & s_2s_1\gtrdot s_1 & -1 & -\eta\\
s_2\to s_2s_1 & s_2s_1\gtrdot s_2 & 0 & 0
\end{array}
\]
Thus the only non-zero cellular $\mathbb{A}^1$-boundary map is the degree-two
differential from the summands indexed by the two length-one cells to
those indexed by the two length-two cells.  With
the target ordered as \((s_1s_2,s_2s_1)\) and the source ordered as
\((s_1,s_2)\), it is represented by
\[
  \partial_2^{\A^1}
  =
  \begin{pmatrix}
    0 & -\eta\\
    -\eta & 0
  \end{pmatrix}.
\]
Equivalently, after permuting one basis, it is the direct sum of two
copies of \(-\eta\).  Thus all non-zero Smith elementary divisors are \(1\), with two
elementary Smith blocks \(\Z\xrightarrow{1}\Z\), both appearing
between the length-two and length-one cells.  The displayed homology is
therefore exactly the specialization of Theorem~\ref{thm:smith-formula}
with two unit elementary summands.  Since \(\partial_3^{\A^1}=0\) and
\(\partial_1^{\A^1}=0\), the cellular homology sheaves are
\[
\begin{aligned}
  \Hcell_0(SL_3/B)&\cong \Z,\\
  \Hcell_1(SL_3/B)&\cong (\KM_1)^2,\\
  \Hcell_2(SL_3/B)&\cong ({}_{\eta}\KMW_2)^2,\\
  \Hcell_3(SL_3/B)&\cong \KMW_3 .
\end{aligned}
\]
After real realization and the degree normalization of
Section~\ref{sec:bruhat-cells}, the two \(-\eta\)-blocks give the
entries \(-2\) in the ordinary cellular boundary matrix.  Hence
\[
  H_1((SL_3/B)(\R),\Z)\cong(\Z/2)^2,\qquad
  H_3((SL_3/B)(\R),\Z)\cong\Z,
\]
with no other reduced integral homology.

\section{Application to the full split flag of type
\texorpdfstring{$F_4$}{F4}}
\label{sec:f4-example}

Let \(X=G/B\) be the full split flag variety of type \(F_4\) over \(k\).
The Weyl exponents are \(1,5,7,11\), so the Bruhat cell-count
polynomial is
\[
  C_{F_4}(t)
  =
  (1+t)(1+t+\cdots+t^5)(1+t+\cdots+t^7)(1+t+\cdots+t^{11}).
\]
The cited Poincar\'e polynomial for the integral free ranks of
\((A_\bullet,D_\bullet)\) is
\[
  B_{F_4}(t)
  =
  (1+t^{11})(1+t^7)(1+t^3)^2.
\]
This is the full-flag \(F_4\) entry in
\cite[Theorem~2.5]{LambertRabelo2026}.
The same integral homology computation records the real flag homology in the
form
\[
  H_i(X_{\R}(\R),\Z)\cong \Z^{\beta_i}\oplus(\Z/2)^{T_i}
\]
\cite[Sections~2 and~6.4]{LambertRabelo2026}; see also the accompanying
implementation \cite{LambertCode}.  For the cellular complex
this is equivalent to all non-zero Smith elementary divisors of
\(D_\bullet\) being \(1\): if \(q\) is a non-zero elementary divisor of
\(D_\bullet\), then the corresponding elementary divisor of the ordinary real
cellular boundary matrix is \(2q\), and it contributes \(\Z/(2q)\) to integral
homology.  Thus the occurrence of only \(\Z/2\)-torsion forces
\(q=1\), and conversely the condition \(q=1\) for every non-zero elementary divisor gives only
\(\Z/2\)-torsion.  Consequently the torsion ranks are determined from
the cell numbers and Betti numbers by
\[
  T_{-1}=0,\qquad T_i=c_i-\beta_i-T_{i-1}.
\]
The resulting values are listed in Table~\ref{tab:f4}.

\begin{table}[htbp]
\caption{Cell numbers, free ranks, and unit Smith-summand counts for the
full split flag variety of type \(F_4\).  Over \(\R\), these are the Betti
numbers and \(2\)-torsion ranks of the real flag manifold.}
\label{tab:f4}
\centering
\begin{tabular}{rrrr}
\toprule
\(i\) & \(c_i\) & \(\beta_i\) & \(T_i\)\\
\midrule
0 & 1 & 1 & 0\\
1 & 4 & 0 & 4\\
2 & 9 & 0 & 5\\
3 & 16 & 2 & 9\\
4 & 25 & 0 & 16\\
5 & 36 & 0 & 20\\
6 & 48 & 1 & 27\\
7 & 60 & 1 & 32\\
8 & 71 & 0 & 39\\
9 & 80 & 0 & 41\\
10 & 87 & 2 & 44\\
11 & 92 & 1 & 47\\
12 & 94 & 0 & 47\\
13 & 92 & 1 & 44\\
14 & 87 & 2 & 41\\
15 & 80 & 0 & 39\\
16 & 71 & 0 & 32\\
17 & 60 & 1 & 27\\
18 & 48 & 1 & 20\\
19 & 36 & 0 & 16\\
20 & 25 & 0 & 9\\
21 & 16 & 2 & 5\\
22 & 9 & 0 & 4\\
23 & 4 & 0 & 0\\
24 & 1 & 1 & 0\\
\bottomrule
\end{tabular}
\end{table}

Applying Corollary~\ref{cor:two-torsion} to these Smith data gives
\[
  \Hcell_i(X)
  \cong
  (\KMW_i)^{\beta_i}
  \oplus
  (\KM_i)^{T_i}
  \oplus
  ({}_\eta\KMW_i)^{T_{i-1}}
  \qquad (1\leq i\leq 24),
\]
with \(\Hcell_0(X)\cong\Z\), where the values of \(\beta_i\) and \(T_i\)
are those in Table~\ref{tab:f4}.  For example,
\[
\begin{aligned}
  \Hcell_1(X) &\cong (\KM_1)^4,\\
  \Hcell_2(X) &\cong (\KM_2)^5\oplus({}_\eta\KMW_2)^4,\\
  \Hcell_3(X) &\cong (\KMW_3)^2\oplus(\KM_3)^9
                  \oplus({}_\eta\KMW_3)^5,\\
  \Hcell_{24}(X) &\cong \KMW_{24}.
\end{aligned}
\]

\section{The full split flag of type \texorpdfstring{$G_2$}{G2}}
\label{sec:g2-example}

Let \(k\) be a perfect field of characteristic different from two,
let \(G/k\) be split and simply connected of type \(G_2\), and set
\(X=G/B\).  We apply Theorem~\ref{thm:eta-comparison} to the
explicit calculation of real cellular boundary matrices in
\cite[\S2.3, Table~2 and the following computation]{RabeloSanMartin}.

Write \(1,2\) for the simple reflections, with \(\alpha_1\) short.
For \(1\leq p\leq5\), order the degree-\(p\) generators by the two
alternating words of length \(p\), beginning with \(1\) and \(2\),
respectively.  In degree six choose \(w_0=121212\).
A word \(a\) labels the Bruhat stratum \(C_{w_0a}\), so its cellular
degree is \(\ell(a)\).  Choose the ordinary cell orientations
from that real calculation, after exchanging the names of the
simple roots and reversing the orientations in degrees two and
five.  Give the motivic cells the corresponding normal orientations
of Section~\ref{sec:bruhat-cells}.  The chain groups are
\[
 C_0=\Z,\qquad C_p=(\KMW_p)^2\ (1\leq p\leq5),
 \qquad C_6=\KMW_6.
\]

\begin{proposition}
\label{prop:g2-homology}
The only nonzero differentials are represented in these bases by
\[
 \partial_2=\eta\begin{pmatrix}0&1\\1&0\end{pmatrix},
 \qquad
 \partial_5=\eta\begin{pmatrix}0&1\\1&0\end{pmatrix}.
\]
The cellular homology sheaves of \(X\) are
\[
\begin{array}{c|c}
 i&\Hcell_i(X)\\ \hline
 0&\Z\\
 1&(\KM_1)^2\\
 2&({}_{\eta}\KMW_2)^2\\
 3&(\KMW_3)^2\\
 4&(\KM_4)^2\\
 5&({}_{\eta}\KMW_5)^2\\
 6&\KMW_6
\end{array}
\]
They vanish in all other degrees.
\end{proposition}

\begin{proof}
The real calculation of \cite[\S2.3]{RabeloSanMartin} has only
\(\delta_2\) and \(\delta_5\) nonzero, both equal to
\(-2\begin{psmallmatrix}0&1\\1&0\end{psmallmatrix}\).
Exchanging the simple-root names preserves these matrices.
Reversing the degree-two and degree-five orientations changes
both signs and leaves all zero differentials unchanged.
Theorem~\ref{thm:eta-comparison} therefore gives the displayed
motivic matrices over every such field \(k\).
Permuting the target bases decomposes the complex into two unit
\(\eta\)-blocks in degrees \((2,1)\), two in degrees \((5,4)\),
and the remaining zero-differential summands.  Taking kernels and
cokernels, with \(\KMW_i/\eta\cong\KM_i\), proves the table.
\end{proof}

In characteristic zero, these matrices also agree with the
augmented cover operators in our subsequent work
\cite[Theorem~7.1 and Proposition~9.1]{LiuLiuTorus}.

\end{document}